\documentclass[11pt,reqno]{amsart}
\usepackage{caption}
\usepackage{enumitem}
 \usepackage{amssymb,amsfonts,amscd,amsbsy, color, amsmath, amsthm, mathtools}
\usepackage[mathscr]{eucal}
\usepackage{url}
\usepackage{graphicx}
\usepackage{mathtools}
\newcommand{\Mod}[1]{\ (\mathrm{mod}\ #1)}
\usepackage{todonotes}
\usepackage{tabularx}
\usepackage{hyperref}
\usepackage{float, color}
\usepackage{placeins}
\usepackage{stmaryrd}
\usepackage{mathabx}
\usepackage{stackengine} 
\usepackage{scalerel}

\newcommand*{\medcap}{\mathbin{\scalebox{1.85}{\ensuremath{\cap}}}}%

\usepackage{scalerel}
\usepackage{stackengine,wasysym}

\newcommand\reallywidetilde[1]{\ThisStyle{%
  \setbox0=\hbox{$\SavedStyle#1$}%
  \stackengine{-.1\LMpt}{$\SavedStyle#1$}{%
    \stretchto{\scaleto{\SavedStyle\mkern.2mu\AC}{.5150\wd0}}{.6\ht0}%
  }{O}{c}{F}{T}{S}%
}}

\def\test#1{
  \reallywidetilde{#1}\,
}

\usepackage{graphicx}

\makeatletter
\newcommand*\bigcdot{\mathpalette\bigcdot@{.75}}
\newcommand*\bigcdot@[2]{\mathbin{\vcenter{\hbox{\scalebox{#2}{$\m@th#1\bullet$}}}}}
\makeatother

\DeclareMathAlphabet{\mathcal}{OMS}{cmsy}{m}{n}
\SetMathAlphabet{\mathcal}{bold}{OMS}{cmsy}{b}{n}

\usepackage{tikz}
\usepackage[labelsep=period]{caption}
\newtheorem{theorem}{Theorem}[section]
\newtheorem{lemma}[theorem]{Lemma}
\newtheorem{proposition}[theorem]{Proposition}
\newtheorem{corollary}[theorem]{Corollary}
\newtheorem{definition}{Definition}[section]

\newtheorem{conjecture}[theorem]{Conjecture}
\numberwithin{equation}{section}

\newcommand{\Z}{\mathbb{Z}}
\newcommand{\Q}{\mathbb{Q}}
\newcommand{\F}{\mathbb{F}}

\usepackage{makecell}

\usepackage{relsize}
\usepackage{stmaryrd}
\usepackage{array}
\newcommand{\sw}[1]{{\color{teal} [{SW: #1}]}}

\usepackage{listings}

\lstdefinestyle{myListingStyle} 
    {
        basicstyle = \small\ttfamily,
        breaklines = true,
    }
\usepackage[nohyperlinks]{acronym}
\renewcommand{\url}[1]{#1}

\makeatletter
\newcommand*{\rom}[1]{\expandafter\@slowromancap\romannumeral #1@}
\makeatother

\makeatletter
\def\imod#1{\allowbreak\mkern5mu({\operator@font mod}\,\,#1)}
\makeatother
\allowdisplaybreaks


\makeatletter
\@namedef{subjclassname@2020}{%
  \textup{2020} Mathematics Subject Classification}
\makeatother

\begin{document}
\title{Indices of nilpotency in certain spaces of modular forms}

\author{Matthew Boylan \and Swati}
\address{Department of Mathematics, University of South Carolina, Columbia, SC, 29208, USA}
\email{boylan@math.sc.edu, s10@email.sc.edu}

\subjclass[2020]{11F11, 11F33}
\keywords{Modular forms mod $p$, index of nilpotency}

\date{}

\begin{abstract}
We study the index of nilpotency relative to certain Hecke operators in spaces of modular forms with integer weight and level $N$ with integer coefficients modulo primes $p$ for $(p, N) \in \{(3, 1), (5, 1), (7, 1), (3, 4)\}$.  In these settings, we prove upper bounds on certain indices of nilpotency.  As an application of our bounds, we prove infinite families of congruences for $p^t$-core partition functions modulo $p$ for $p\in \{3, 5, 7\}$ and $t\geq 1$, and we prove an infinite family of congruences modulo $3$ for the $r$th power partition function, $p_r(n)$, when $r = 12k$ with $\gcd(k,6) = 1$. We also include conjectures on a function which quantifies degree lowering on powers of the Delta function by the relevant Hecke operators in these settings, and on the index of nilpotency relative to a modification of this degree-lowering function.
\end{abstract}

\maketitle

\section{Introduction and Statement of Results}

\subsection{Introduction}
\noindent
Let $p$ be prime, let $\Q_{(p)}$ denote the localization of $\Q$ at $p$, let $N\geq 1$, let $k\geq 0$, and let
\[
M_k(\Gamma_0(N))_{(p)} = 
M_k(\Gamma_0(N)) \cap \Q_{(p)}\llbracket q \rrbracket,
\]
denote the space of weight $k$ modular forms on $\Gamma_0(N)$ with
coefficients in $\Q_{(p)}$.
If $f = \sum a_n q^n \in \Q_{(p)}\llbracket q\rrbracket$ then we define

\[
\widetilde{f} = \sum \widetilde{a_n}\, q^n \in \F_p\llbracket q\rrbracket,
\]
to be its coefficient-wise reduction modulo $p$, and we define
\[
\widetilde M^{(p)}(\Gamma_0(N)) =
\{\widetilde{f} : f \in M_k(\Gamma_0(N))_{(p)}\}
\subseteq \F_p\llbracket q\rrbracket,
\]
the algebra of holomorphic integer weight modular forms on $\Gamma_0(N)$ with coefficients in $\Q_{(p)}$ reduced modulo $p$.  For relevant details on these spaces, see Section \ref{background}.  Theorem \ref{nilpotent_theorem} below gives instances where Hecke operators act locally nilpotently on subspaces of $\widetilde{M}^{(p)}(\Gamma_0(N))$.  When a Hecke operator $T$ acts locally nilpotently on a subspace $\widetilde{M} \subseteq \widetilde{M}^{(p)}(\Gamma_0(N))$, we define, for every nonzero $\widetilde{f} \in \widetilde{M}$, the index of nilpotency of $\widetilde{f}$ with respect to $T$ by
\begin{equation} \label{index_def}
N_T(\widetilde f, p) := \min\{u \geq 1 : \widetilde f \mid T^u = 0 \: \text{in} \: \F_p\llbracket q\rrbracket\}.
\end{equation}
In this paper, we study $N_T(\widetilde{f}, p)$ in certain settings, and we give applications of our results.    
Some recent works on local nilpotency include 
\cite{coss_zhou}, \cite{gg}, \cite{medthesis}, \cite{medved}, \cite{monsky}, \cite{moon_taguchi}, \cite{N-S-1}, \cite{N-S-2}, \cite{ono_taguchi}; some recent applications of nilpotency to congruence properties of arithmetic functions include \cite{boylan}, \cite{boylan_imrn}, \cite{boylan_ono}, \cite{chen1}, \cite{chen2}, \cite{dai}, \cite{mahlburg}, \cite{smith_ye}.  

Before we describe our work, we give some of the benchmark results on local nilpotency of Hecke operators on spaces $\widetilde{M}^{(p)}(\Gamma_0(N))$.  We require some notation.  We consider the set of primes 
 \[
 L_{p, N} = \{ \ell: \ell \equiv \pm 1 \Mod{p} \: \text{is prime and} \: \ell \nmid N\},
 \]
 and for all primes $\ell\in L_{p, N}$, 
\noindent
we consider the modified Hecke operators  
\begin{gather} \label{mod_Hecke}
	\displaystyle{	T_{\ell}^{'} = 
		\begin{dcases}
			T_{\ell} &\text{if $\ell \equiv -1 \Mod{p}$},\quad \\
			 T_{\ell} - 2 \quad &\text{if $\ell \equiv 1 \Mod{p}$}. \quad
		\end{dcases} 
  }
\end{gather}
We let $\widetilde{M}^{0, (p)}(\Gamma_0(N))$ denote the subspace of $\widetilde{M}^{(p)}(\Gamma_0(N))$ of forms with weights congruent to zero modulo $p - 1$. Theorem 3 of \cite{swd} implies, for $p\in \{2, 3\}$, that $\widetilde{M}^{0, (p)}(\Gamma_0(1)) = \widetilde{M}^{(p)}(\Gamma_0(1)) = \mathbb{F}_p[\widetilde{\Delta}]$. Further, Theorem 2(iv) of \cite{swd} implies that $\widetilde{M}^{(5)}(\Gamma_0(1)) = \F_{5}[\widetilde{E_{6}}]$.  Since $\widetilde{\Delta} = 2 - 2 \widetilde{E}_6^2$, we have $\widetilde{M}^{0,(5)}(\Gamma_0(1)) = \F_{5}[\widetilde{E}_{6}^2] = \F_{5}[\widetilde{\Delta}]$. Similarly, we observe that $\widetilde{M}^{(7)}(\Gamma_0(1)) = \F_{7}[\widetilde{E_{4}}]$ and hence, that $\widetilde{\Delta} = 1 - \widetilde{E}_{4}^3$ gives $\widetilde{M}^{0, (7)}(\Gamma_0(1)) = \F_{7}[\widetilde{E}_{4}^3] = \F_{7}[\widetilde{\Delta}]$. The following theorem captures some instances where Hecke operators act locally nilpotently on spaces of modular forms modulo~$p$.

\begin{theorem} \label{nilpotent_theorem} In the notation above, we have
\begin{enumerate} 
    \item  ({\it Serre,\,Tate} \cite{ser3, tate}) {\it Let $p \in \{2, 3, 5, 7\}$. For all $\ell \in L_{p,1}$, the operator $T_{\ell}^{'}$ acts locally nilpotently on $\widetilde{M}^{0,(p)}(\Gamma_0(1)) = \mathbb{F}_p[\widetilde{\Delta}]$.}
    \item  ({\it Ono,\,Taguchi} \cite{ono_taguchi}) {\it Let $1 \leq N \leq 17$ be odd, and let $a \geq 0$. For all primes $\ell \nmid 2N$, the operator $T_{\ell}^{'}$ acts locally nilpotently on $\widetilde{M}^{(2)}(\Gamma_{0}(2^a N))$.}
    \item  ({\it Moon,\,Taguchi} \cite{moon_taguchi}) {\it Let $(p, N) = \{ (3, 4), (5, 2)\}$, and let $a \geq 0$. For all $\ell \in L_{p, N}$, the operator $T_{\ell}^{'}$ acts locally nilpotently on $\widetilde{M}^{(p)}(\Gamma_{0}(p^a N))$.}
\end{enumerate}
\end{theorem}

We now give theorems of Nicolas and Serre \cite{N-S-2} and Medvedovsky~\cite{medthesis} on indices of nilpotency.  To state the theorem of Nicolas and Serre, we need to extend the definition \ref{index_def}.  When an algebra $A$ of Hecke operators acts locally nilpotently on a subspace $\widetilde M \subseteq \widetilde M^{(p)}(\Gamma_0(N))$, we
define $N_{A}(\widetilde{f}, p)$, the index of nilpotency of a nonzero $\widetilde{f} \in \widetilde{M}$ with respect to $A$, to be the integer
$n \geq 1$ such that
\begin{enumerate}
\item for all $R_1,\ldots,R_{n} \in A$, we have
\[
\widetilde f \mid R_1 \mid \cdots \mid R_{n} = 0
\]
\item there exists $S_1,\ldots,S_{n - 1} \in A$ with
\[
\widetilde f \mid S_1 \mid \cdots \mid S_{n - 1} \neq 0.
\]
\end{enumerate}
For all $T\in A$ and nonzero $\widetilde{f} \in \widetilde{M}$, we note that $N_{T}(\widetilde{f}, p)\leq N_{A}(\widetilde{f}, p)$.  
To simplify notation, when $\ell \in L_{p, N}$ and $\widetilde{f}$ is in one of the spaces in Theorem~\ref{nilpotent_theorem}, we set $N_{\ell}(\widetilde{f}, p) = N_{T_{\ell}^{\prime}}(\widetilde{f}, p)$.   
The following theorem gives bounds on indices of nilpotency modulo $p\in \{2, 3, 5, 7\}$ in level one. 
\begin{theorem} \label{known} Let $k\geq 1$.  
\begin{enumerate} 
\item ({\it Nicolas-Serre (2012)} \cite{N-S-2})
{\it Let $k$ be odd with base 2 expansion $k = \sum \beta_{i}(k) 2^{i}$,  and define 
\begin{equation*}
n_{3}(k) = \sum \beta_{2i+1}(k) 2^{i}, \ \ n_{5}(k) = \sum \beta_{2i+2}(k) 2^{i},
\end{equation*}}
and $A = \langle T_{\ell} : \ell\neq 2\rangle$.
Then we have
\begin{equation*} 
\frac{1}{2} \sqrt{k} < N_A(\widetilde{\Delta}^{k}, 2) = 1 + n_{3}(k) + n_{5}(k) < \frac{3}{2}\sqrt{k}.
\end{equation*}
\item ({\it Medvedovsky (2015)} \cite{medthesis})
{\it For all $(p,\ell)$ listed in the table below, the explicit constants $c > 0$ and $0<\alpha < 1$ given there satisfy $N_{\ell}(\widetilde{\Delta}^{k}, p) < c \cdot k^{\alpha}$.}

\renewcommand{\arraystretch}{1.1} 
\setlength{\tabcolsep}{12pt} 
\begin{table}[ht]
  \centering
\begin{tabular}{|c|c|c|||c|c|c|c|}
\hline
$(p, \ell)$ & $c$ & $\alpha$ & $(p, \ell)$ & $c$ & $\alpha$  \\ 
\hline
(2, 3) & 3 & 0.5 & (2, 5) & 7/3 &  2/3 \\
\hline
(3, 2) & 4 & $\log_{3}(2)$ & (3, 7) & 3.2 &  $\log_{9}(6)$ \\
\hline
(5, 19) & 138/11 &  $\log_{25}(23)$ & (5, 11) & 6.3 &  $\log_{25}(21)$ \\
\hline
(7, 13) & 564/23 & $\log_{49}(47)$ & (7, 29) & 564/23 & $\log_{49}(47)$ \\
\hline
\end{tabular}
\caption{Explicit values of $c$ and $\alpha$ for different $p$ and $\ell$}
\label{tab: explicit}
\end{table}
\end{enumerate}
\end{theorem}

\subsection{Statement of Results}
We now turn to our work.  For details on modular forms, see Section~\ref{background}.  
Our first theorem gives upper bounds on $N_{\ell}(f^k, p)$ in certain cases.  

\begin{theorem}\label{simple_bound}
Let $k\geq 1$.  
\begin{enumerate} 
\item For all primes $\ell\neq 3$, we have
\begin{equation*}
N_{\ell}(\widetilde{\Delta}^k, 3) \leq \begin{cases}
1 + \big\lfloor\frac{k}{3}\big\rfloor, & \ell\equiv 1 \pmod{3}, \\
1 + \big\lfloor\frac{2k}{3}\big\rfloor, & \ell\equiv 2 \pmod{3}. \\
\end{cases}
\end{equation*}
\item Let $p\in \{5, 7\}$. For all primes $\ell\equiv \pm 1 \pmod{p}$, we have 
\begin{equation*}
N_{\ell}(\widetilde{\Delta}^k, p) \leq \begin{cases}
1 + \big\lfloor\frac{2k}{5}\big\rfloor, & p = 5, \, p\nmid k, \\
1 + \big\lfloor{\frac{3k}{7}\big\rfloor}, & p = 7, \, k \in \{1, 3, 5\} \pmod{7}, \\
2 + \big\lfloor{\frac{3k}{7}\big\rfloor}, & p = 7, \, k \in \{2, 4, 6\} \pmod{7}, \\
N_{\ell}(\widetilde{\Delta}^{k/p}, p), & p\mid k.
\end{cases}
\end{equation*}
\item {\it Let $D_2(z) = \eta(2z)^{12}\in S_6(\Gamma_0(4))$, where $\eta(z)$ is the Dedekind eta-function as in \eqref{eta}, and let $\ell \not\in\{2, 3\}$ be prime.}  
\begin{enumerate} 
\item Suppose that $\gcd(k, 6) = 1$. We have 
\begin{equation*}
N_{\ell}(\widetilde{D}_2^k, 3) \leq  1 + \Big\lfloor\frac{k}{3}\Big\rfloor.
\end{equation*}
\item We also have 
\begin{equation*}
N_{\ell}(\widetilde{D}_2^k, 3) \leq \begin{cases} N_{\ell}(\widetilde{D}_2^{k/3}, 3), & 3\mid k, \\
N_{\ell}(\widetilde{\Delta}^{k/2}, 3), & 2\mid k.
\end{cases}
\end{equation*}
\end{enumerate}
\end{enumerate}
\end{theorem}
\noindent
{\bf Remarks.}
\begin{enumerate}
\item An important feature of Medvedovsky's bounds on $N_{\ell}(\widetilde{\Delta}^k, p)$ in part two of Theorem~\ref{known} is that they are {\it sublinear}.  While ours are linear, we note that they are stronger than the bounds in Table \ref{tab: explicit} for a large--but finite--number of values of $k$, as follows:
\vspace{1mm}
\begin{itemize}
    \item For $\ell \equiv 1 \pmod{3}$ and $k \leq 2.1 \times 10^5$,
    \item For $\ell \equiv -1 \pmod{5}$ and $k \leq 5.86 \times 10^{57}$,
    \item For $\ell \equiv 1 \pmod{5}$ and $k \leq 1.27 \times 10^{22}$,
    \item For $\ell \equiv \pm 1 \pmod{7}$ and $k \leq 1.36 \times 10^{164}.$
\end{itemize}


\item Let $p\in \{2, 3, 5, 7\}$ and let $\ell\in \{\pm 1 \pmod{p}\}$ be prime. By Th\'eor\`eme 23 of \cite{N-S-2} and Propositions 4.21 and 6.2 of \cite{medthesis}, the sequence $\{\widetilde{\Delta}^{k}\mid T_{\ell}' : k\geq 1\}$ satisfies a recurrence
\begin{equation} 
\label{recurrence}
    \widetilde{\Delta}^k\mid T_{\ell}' = \sum_{1\leq m\leq t} b_{m, \ell}(\widetilde{\Delta})\widetilde{\Delta}^{k - m}\mid T_{\ell}^{'},
\end{equation}
where $t = \begin{cases} \ell + 1, & \ell \equiv -1\pmod{p} \\
\ell + 2, & \ell\equiv 1 \pmod{p}\end{cases}$ is the order of the recurrence
and $b_{m, \ell}\in \mathbb{F}_p[Y]$ has degree at most $m$.  For all $\widetilde{f}\in \mathbb{F}_p[\widetilde{\Delta}]$, we let $\text{deg}_{\widetilde{\Delta}}(\widetilde{f})$ denote the degree of $\widetilde{f}$ in $\widetilde{\Delta}$.  In the terminology of Chapter 4 of \cite{medthesis}, the sequence is $i$-filtered for some $i\geq 0$ precisely when we have $\text{deg}_{\widetilde{\Delta}}(\widetilde{\Delta}^n\mid T_{\ell}')\leq n - i$ for all $n\geq 1$. Part(1) of Theorem \ref{nilpotent_theorem} implies that the sequences $\{\widetilde{\Delta}^k\mid T_{\ell}' : k\geq 1\}$ are $1$-filtered.  Parts (1) and (2) of Theorem \ref{simple_bound} follow from showing that the sequences are $2$-filtered when $p\in \{3, 5, 7\}$.  

When $p\in \{3, 5, 7\}$, we note that for certain primes $\ell$, the sequence $\{\widetilde{\Delta}^{k}\mid T_{\ell}' : k\geq 1\}$ may be $i$-filtered for some $i > 2$, which would yield better bounds on $N_{\ell}(\widetilde{\Delta}^k, p)$ than in the theorem.  For example, when $p = 5$, computations show that $\text{deg}_{\widetilde{\Delta}}(\widetilde{\Delta}^{n}\mid T_{11}')\leq n - 4$ for $1\leq n\leq 13$.  An induction argument using \eqref{recurrence} shows that the sequence $\{\widetilde{\Delta}^{k}\mid T_{11}' : k\geq 1\}$ is $4$-filtered.  An argument as in the proof of the theorem then gives $N_{11}(\widetilde{\Delta}^k, 5)\leq 1 + \left\lfloor\frac{k}{4}\right\rfloor$.
 
\item Medvedovsky's computations and ours suggest that $N(\widetilde{\Delta}^k, 3) = O(k^{1/2})$.  Our computations also suggest that $N_{11}(\widetilde{\Delta}^k, 5)$ and $N_{19}(\widetilde{\Delta}^k, 5) = O(k^{1/2})$.
\end{enumerate}

\medskip

Next, we let $\delta\mid 24$ and $D_{\delta}(z) = \eta(\delta z)^{24/\delta}$, and we note that $D_1(z) = \Delta(z)$.  By Propositions 5.9.2 and 5.9.3 of \cite{Cohen2017ModularFA}, we have \begin{equation*} 
    D_{\delta}(z)\in \begin{cases} 
    S_{1/2}\left(\Gamma_0(576), \left(\frac{12}{\bigcdot}\right)\right), & \delta = 24, \\
    S_{3/2}\left(\Gamma_0(64), \left(\frac{-4}{\bigcdot}\right)\right), & \delta = 8, \\
    S_{12/\delta}\left(\Gamma_0(\delta^2), \left(\frac{(-1)^{12/\delta}\delta^{24/\delta}}{\bigcdot} \right)\right), & \delta \not\in\{8, 24\},
\end{cases}
\end{equation*}
where the forms with half-integer weight transform with respect to the theta-multiplier system as in Shimura's theory. 
In the following proposition, we identify some primes $p$ and $\ell$ and exponents $k$ such that $\widetilde{D}_{\delta}^k(z) \mid T_{\ell} = 0$ in $\F_{p}\llbracket q\rrbracket$.  In such cases, when $T_{\ell}$ acts nilpotently on an algebra containing $\widetilde{D}_{\delta}$ as in Theorem \ref{nilpotent_theorem}, we have $N_{\ell}(\widetilde{D}_{\delta}^k(z), p) = 1$.  

\begin{proposition} \label{vanishing}
Let $m\geq 1$, let $p \neq \ell$ be primes, and let $\left(\frac{\bigcdot}{\ell}\right)$ denote the Legendre symbol.  
\begin{enumerate}
\item 
Let $\delta\mid 24$, let $m$ be odd, let $p$ satisfy
\begin{enumerate}
    \item $p + 1\mid 24/\delta$,
    \item $\frac{24/\delta}{p + 1} = p^t$ or $3p^t$ for some $t\geq 0$, 
    \item $p\nmid \delta$,
\end{enumerate}
and let $r_{p, m} = \frac{p^m + 1}{p + 1}\in \Z$.  Then for all primes $\ell \neq p$ with $\ell \nmid \delta$ and $\left(\frac{-p}{\ell}\right) = -1$, we have $\widetilde{D}_{\delta}(z)^{r_{p,m}}\mid T_{\ell} = 0$ in $\F_{p}\llbracket q \rrbracket$.  The set of $(\delta, p)$ satisfying conditions (a), (b), (c) is 
\begin{equation*}
    \{(1, 2), (1, 7), (1, 23), (2, 3), (2, 11), (3, 7), (4, 5)\}.
\end{equation*}
\item Let $d\mid 24$ and let $p$ satisfy
\begin{enumerate}
    \item $24/d = p^t$ or $3p^t$ for some $t\geq 0$, 
    \item $p\nmid d$.
\end{enumerate}
Then for all primes $\ell \neq p$ with $\ell\nmid d$ and 
$-1 = \left(\frac{-p^m}{\ell}\right) = \begin{cases} \left(\frac{-p}{\ell}\right), & m \ \text{is odd}, \\ \left(\frac{-1}{\ell}\right), & m \ \text{is even}
\end{cases}$, we have 
$\widetilde{D}_{d}(z)^{p^m + 1}\mid T_{\ell} = 0$
in $\F_{p}\llbracket q\rrbracket$.  The set of $(d, p)$ satisfying conditions (a) and (b) is 
\begin{equation*}
    \{(1, 2), (3, 2)\}\cup \{(8, p) : p\geq 3\} \cup \{(24, p) : p\geq 5\}.
\end{equation*}
\end{enumerate}
\end{proposition}
\noindent
{\bf Remark.}
The primes $p \neq 2$ in part one of the proposition are precisely those with $\frac{24/\delta}{p + 1} \in \{1, 3\}$, in which cases we find that $D_{\delta}(z)^{r_{p, m}}$ is a product of two theta-series modulo $p$: 
\begin{equation*}
D_{\delta}(z)^{r_{p,m}} \equiv \eta(\delta z)^{\frac{24/\delta}{p+1}}\eta(\delta p^mz)^{\frac{24/\delta}{p + 1}} \pmod{p}.
\end{equation*}

\medskip

Our next theorems are applications of Theorem \ref{simple_bound} to congruences for arithmetic functions.
Let $t\geq 1$.  We recall that a partition $\lambda$ of a positive integer $n$ is $t$-core precisely when none of the hook numbers associated to its Ferrer's-Young diagram are divisible by $t$. The notion of a $t$-core partition plays a role in modular group representation theory \cite{jk}.  We denote the number of $t$-core partitions of $n$ by $a_{t}(n)$.  See \cite{survey_tcore} for a survey of results on $t$-core partitions and $t$-core partition functions.  
For $p\in \{3, 5, 7\}$, our first application of Theorem \ref{simple_bound} yields congruences for $p^t$-core partition functions $a_{p^t}(n)$ modulo~$p$. 

\medskip

\begin{theorem} \label{partition_5}
Let $t \geq 1$, let $p\in \{3, 5, 7\}$, let 
\begin{equation*}
k_{p, t} = \begin{cases}
\frac{9^t - 1}{8}, & p = 3, \\
\frac{p^{2t} - 1}{24}, & p \in\{5, 7\},
\end{cases} 
\end{equation*}
and let 
\begin{equation*}
m_{p, t} = \begin{cases}
1 + \big\lfloor\frac{2 k_{3,t}}{3}\big\rfloor = 1 + 3 \left(\frac{9^{t-1} - 1}{4}\right), & p = 3 \:\: \text{and} \:\: \ell \equiv 2 \pmod{3}, \\
1 + \big\lfloor\frac{k_{3,t}}{3}\big\rfloor = 1 + 3 \left(\frac{9^{t-1} - 1}{8}\right), & p = 3 \:\: \text{and} \:\: \ell \equiv 1 \pmod{3}, \\
1 + \big\lfloor\frac{2k_{5,t}}{5}\big\rfloor = 1 + 5\left(\frac{25^{t-1} - 1}{12}\right), & p = 5 
,\\
2 + \big\lfloor\frac{3k_{7,t}}{7}\big\rfloor = 2 + 7\left(\frac{49^{t-1} - 1}{8}\right), & p = 7. \\
\end{cases}
\end{equation*}
\begin{enumerate}
\item Let $p\in \{3, 5, 7\}$, and let $\ell\equiv -1 \Mod{p}$ be prime.  Then for all $r$ with $p^r \geq m_{p, t}$, and for all $n\geq 1$ with $\ell\nmid n$, we have 
\begin{equation*}
a_{p^t}(\ell^{p^r}n - k_{p, t}) \equiv - a_{p^t}(\ell^{p^r - 2}n - k_{p,t}) \pmod{p}.
\end{equation*}
\item Let $p\in \{5, 7\}$, and let $\ell_1,\ldots, \ell_{m_{p, t}}\equiv -1\Mod{p}$ be distinct primes.  Then for all $n\geq 1$ with $\gcd(\ell_1\cdots\ell_{m_{p, t}}, n) = 1$, we have
\begin{equation*}
a_{p^{t}}(\ell_1\cdots\ell_{m_{p, t}}n - k_{p, t}) \equiv 0 \pmod{p}.
\end{equation*}
\item Let $p\in \{3, 5, 7\}$, and let $\ell \equiv 1 \pmod{p}$. Then for all $r$ with $p^r \geq m_{p,t}$, and for all $n\geq 1$ with $\ell\nmid n$, we have 
\begin{equation*}\label{item:cong5_11}
a_{p^t}(\ell^{p^r} n - k_{p,t}) - a_{p^t}(\ell^{p^r - 2}n - k_{p,t}) \equiv 2 a_{p^{t}} \left(n - k_{p,t} \right) \pmod{p}.
\end{equation*}
\end{enumerate}  
\end{theorem}

{\bf Remarks.}
\begin{enumerate}
\item 
The integers $m_{p, t}$ in the theorem are bounds on $N_{\ell}(\Delta^{k_{p, t}}, p)$ from parts (1) and (2) of Theorem ~ \ref{simple_bound} when $p = 3$ and $p\in \{5, 7\}$, respectively.
\item
Analogues of parts (1) and (2) of the theorem hold for $p\in \{2, 3\}$.~Using part (1) of Theorem \ref{nilpotent_theorem}, the first author \cite{boylan} proved congruences of these types for $p = 2$ prior to the result of Nicolas and Serre in part (1) of Theorem \ref{known}.  Later, Chen \cite{chen2} used part~(1) of Theorem \ref{known} to sharpen the results in \cite{boylan}, settling a conjecture from \cite{H-S}.  Dai \cite{dai} used part (2) of Theorem \ref{known} to prove the analogue of part (2) for $p = 3$.  
\end{enumerate}

\medskip

We now let $r\geq 1$, and we define functions $p_r(n)$ by
\begin{equation} \label{gen_fun1}
\sum_{n \geq 0} p_{r}(n) q^{n} = \prod_{n \geq 1} \: (1 - q^{n})^r.
\end{equation}
We let $p_{e, r}(n)$ and $p_{o, r}(n)$ denote the number of $r$-colored partitions of $n\geq 1$ into even and odd numbers of parts, respectively, and we note that $p_r(n) = p_{e,r}(n) - p_{o, r}(n)$.  
Our second application of Theorem \ref{simple_bound} gives congruences for $p_{12r}(n) \Mod{3}$ when $\gcd(r, 6) = 1$.
\begin{theorem} \label{partition_3} 
Let $r\geq 1$ with $\gcd(r, 6) = 1$, let $u_r = 1 + \left\lfloor\frac{r}{3}\right\rfloor$, and let $\ell \neq 2$ be prime.  
\begin{enumerate}
\item Let $\ell \equiv -1 \pmod{3}$. Then for all $j$ with $3^{j}\geq u_{r}$ and for all $n \geq 1$ with $\ell\nmid 2n$, we have
\begin{equation*}
p_{12r} \left( \frac{\ell^{u_r}n - r}{2}  \right) \equiv \, 2p_{12r}\left(\frac{\ell^{u_r - 2}  n -r}{2}\right) \Mod{3}.
\end{equation*}
\item Let $\ell_1,\ldots, \ell_{u_{r}}\equiv -1\Mod{3}$ be distinct odd primes.  Then for all $n\geq 1$ relatively prime to $2\ell_1\cdots\ell_{u_r}$, we have
\begin{equation*}
p_{12r}\left(\frac{\ell_1\cdots\ell_{u_{r}} n - r}{2} \right) \equiv 0\Mod{3}.
\end{equation*}
\item Let $\ell \equiv 1 \pmod{3}$.  Then for all $j$ with $3^{j}\geq u_{r}$ and for all $n \geq 1$ with $\ell\nmid 2n$, we have
\begin{equation*}
p_{12r} \left( \frac{\ell^{3^{j}}n - r}{2}  \right)
- p_{12r}\left(\frac{\ell^{3^{j} - 2}n - r}{2}\right)
\equiv  2 p_{12r} \left( \frac{n - r}{2}  \right) \Mod{3}.
\end{equation*}
\end{enumerate}  
\end{theorem}

\medskip

\subsection{Computations and Conjectures on Some Related Quantities}
Here we discuss quantities of interest that we encountered in our study of nilpotency, and we make some conjectures on these quantities.  We let $p\in \{5, 7\}$, and for all $k\geq 1$ and all primes $\ell\equiv\pm{1}~\pmod{p}$, we define
\begin{equation}\label{deg_lower}
{D_{\ell}(k, p) = \text{deg}_{\widetilde{\Delta}}(\widetilde{\Delta}^k\mid T_{\ell}^{'})},
\end{equation}
where $T_{\ell}^{'}$ denotes the modified Hecke operator as in \eqref{mod_Hecke}, and $\text{deg}_{\widetilde{\Delta}}(f)$ denotes the degree of $\widetilde{f} \in \mathbb{F}_p[\widetilde{\Delta}]$ in $\widetilde{\Delta}$.

\subsubsection{Some Conjectures Modulo $5$}
When $\ell = 19$ and $p = 5$, we define the modified degree-lowering function 
\begin{equation*}
D'_{19}(k, 5) = \begin{cases} D_{19}(k, 5) - 1, & 5\mid D_{19}(k, 5) \\
D_{19}(k, 5), & \text{otherwise}.
\end{cases}
\end{equation*}
For all $t\geq 1$, we denote by $D'^{(t)}_{\ell}(k, p)$ the $t$th iteration of $D_{\ell}(k, p)$.  The first part of Theorem~\ref{nilpotent_theorem} implies that $D_{\ell}'(k,p) < k$ and that there exists a least $j\geq 1$ with $D'^{(j)}_{\ell}(k,p) = -\infty$.
Our computations suggest precise conjectures on $D_{19}(k, 5)$ and on a second quantity, $S'_{19}(k, 5)$, the nilpotence index of $D_{19}^{'}(k, 5)$: 
\[
S_{19}'(k, 5) = \text{min}\left\{t\geq 1 : D_{19}'^{(t)}(k, 5) = -\infty\right\}.
\]
Furthermore, our computations suggest that $N_{19}(\widetilde{\Delta}^k,5) \leq S_{19}'(k, 5)$.  For $n\geq 1$, we define $v_5(n)$ to be the largest power of $5$ dividing $n$.  

\medskip

\begin{conjecture} \label{conj}
Assume the notation above.  
\begin{enumerate}
\item Let $k\geq 1$ with $5\nmid k$.  Then the following formulas hold for $D_{19}(k, 5)$: 
\renewcommand{\arraystretch}{1.5} 
\setlength{\tabcolsep}{12pt} 
\begin{table}[ht] 
  \centering
\begin{tabular}{|c|c|c|c|} 
\hline
$k$ & $D_{19}(k, 5)$ \\ 
\hline
$k \equiv j \Mod{5}, j \in \{1,2\}$ & $k - \left( \frac{5^{v_{5}(k - j)} + 1}{3} \right), v_{5}(k - j) \equiv 1 \Mod{2}$ \\
& $k - \left( \frac{5^{v_{5}(k - j)} + 5}{3} \right), v_{5}(k - j) \equiv 0 \Mod{2}$ \\
\hline
$k \equiv 3,4 \Mod{5}$, $k \not\in \{13, 14\} \Mod{25}$ & $k - 2$ \\
\hline
$k \equiv 13 \Mod{25}$ & $k - 3$ \\
\hline
$k \equiv 14 \Mod{25}$ & $k - 4$ \\
\hline
\end{tabular}
\caption{Degree $D_{19}(k, 5)$ of $\widetilde{\Delta}^k \mid T_{19}$ in $\mathbb{F}_5\llbracket q\rrbracket$}
\label{tab: degree}
\end{table}
\item Suppose that $k$ has base $5$ expansion $k = \sum a_i\cdot 5^i$, and that $s = a_0 + a_1 \cdot 5$.  For all $t\geq 0$, let $c_t = S_{19}^{\prime}(5^t + 1, 5) - 1$. Then we have
\begin{itemize}
\item $c_0 = 0$, $c_1 = 2$, and $c_i = 3c_{i - 1} + 2c_{i - 2}$ for all $i\geq 2$, and 
\item $S_{19}^{\prime}(k, 5) = S_{19}^{\prime}(s, 5) + \underset{i \geq 2} \sum a_ic_i$.
\end{itemize}
\item We have $S_{19}^{\prime}(k, 5) = O(k^{\alpha})$, 
where $\alpha = \log_5\left(\frac{3 + \sqrt{17}}{2}\right) \approx 0.7892$.  
\item We have $N_{19}(\widetilde{\Delta}^k, 5) \leq S_{19}^{\prime}(k, 5).$
\end{enumerate}
\end{conjecture}

\bigskip

\noindent
{\bf Remarks.} 
\begin{enumerate}
\item Using PARI/GP \cite{PARI2}, we verified the conjecture for $k\leq 8 \times 10^4$.
\item The quadratic irrational $\frac{3 + \sqrt{17}}{2}$ in the argument of the logarithm in part (3) is the dominant root of the polynomial $x^2 - 3x - 2$, which is the characteristic polynomial of the recurrence in part (2).  
\item Following Nicolas and Serre \cite{N-S-2}, one approach to proving the formulas in Table \ref{tab: degree} might be to argue by induction using the recurrence satisfied by $\widetilde{\Delta}^k \mid T_{19}$ as in \eqref{recurrence}. As in Chapter 6 of \cite{medthesis}, one can explicitly write down this order $20$ recurrence and its associated characteristic polynomial.  However, its relatively large number of terms and the complexity of its coefficients prevented us from effectively pursuing this approach to prove the conjectured formulas.  
\item We conjecture a more complicated formula for $D_{11}(k, 5)$, which we omit for brevity.
\end{enumerate}

\subsubsection{Computations Modulo $7$}
Our computations suggest a similar formula for the nilpotence index of a modified degree-lowering function in the mod $7$ setting, which we briefly describe.  With $D_{29}(k, 7)$ as in \eqref{deg_lower}, we define $D_{29}''(k,7)$ to be the degree of the second highest nonzero term in $\widetilde{\Delta}^k\mid T'_{29}\in \mathbb{F}_7[\widetilde{\Delta}]$ if $7\mid D_{29}(k, 7)$ or if $D_{29}(k, 7)\equiv 5 \ (\text{mod} \ 98)$; otherwise, we define $D''_{29}(k, 7)$ to be $D_{29}(k, 7)$.   
We also define
\[
S''_{29}(k,7) = \text{min}\{s : D''^{(s)}_{29}(k, 7) = -\infty\}.
\]
\noindent
We suppose that $7 \nmid k$, that $k$ has base $7$ expansion $k = \sum a_i\cdot 7^i$, and that $s\equiv k\pmod{98}$ with $1\leq s\leq 97$.  For all $t\geq 1$, we let $y_{t} = S_{29}^{\prime\prime}(2\cdot 7^{t} + 1, 7) - 1$. Then for $1 \leq k\leq 3.5 \times 10^4$ with certain exceptions depending on the residue class of $k \pmod{2\cdot 7^j}$ for $j\geq 2$, our computations give 
\[
S_{29}^{\prime\prime}(k, 7) = S_{29}^{\prime\prime}(s, 7) + \underset{i \geq 2}\sum x_{i}\,y_{i}, 
\]
where the $x_{i}$ are linear combinations of $\lfloor \frac{a_{i}}{2} \rfloor$ and $\lfloor \frac{a_{i} + 7}{2} \rfloor$ depending on the parity of $a_{i}$.  
For brevity, we give one example of an exception.  When $k \equiv 5 \Mod{98}$, we have $S_{29}^{"}(\Delta^k, 7) = 1 + S_{29}^{"}(\Delta^s, 7) + \displaystyle{\sum_{i \geq 2}  x_{i} y_{i}}$.

Our computations also suggest that $y_{i}$ satisfies the recurrence relation
\[
y_{1} = 3, \, y_{2} = 16, \, \text{and} \, y_{i} = 5\,y_{i - 1} + 2\,y_{i - 2} \ \text{for $i \geq 3$}; 
\]
that $S_{29}^{\prime\prime}(k, 7) = O(k^{\alpha})$, where $\alpha = \log_{7}\left( \frac{5 + \sqrt{33}}{2} \right)\approx 0.864$; and that $N_{29}(\widetilde{\Delta}^k, 7) \leq S_{29}^{\prime\prime}(k, 7)$. The quadratic irrational argument in the logarithm arises as the dominant root of the characteristic polynomial for the recurrence satisfied by $\{y_i\}$.  We note that a similar recurrence relation can be conjectured for $\ell = 13$, but we omit it due to the complicated definition of $D_{13}''(k,7)$. 
 
\subsubsection{Computations Modulo $3$}
We recall that $D_2(z) = \eta(2z)^{12} \in S_6(\Gamma_0(4))$.  Part 3 of Theorem~\ref{simple_bound} gives bounds on 
$N_{\ell}(\widetilde{D_2}^k, 3)$ for all $k\geq 1$ and for primes $\ell\not\in\{2, 3\}$.  While our proof does not require $\langle T_{\ell} : \ell\geq 5\rangle$ - invariance of the $\mathbb{F}_3$-module generated by $\{D_2^k : \gcd(k, 6)\}$, computations with $k\leq 4\times 10^4$ and small $\ell$ suggest that it is, indeed, invariant.  Under this assumption, for all $k$ and for all primes $\ell\geq 5$, we define 
\[
E_{\ell}(k, 3) = \deg_{\widetilde{D_{2}}}(\widetilde{D_{2}}^{k} \mid T_{\ell}')
\]
and its modification
\begin{equation*}
E_{\ell}^{\prime}(k, 3) :=
\begin{cases}
E_{\ell}(k, 3) - 1, & 3\mid E_{\ell}(k, 3), \\
E_{\ell}(k, 3), & \text{otherwise},
\end{cases}
\end{equation*}
and we define $S'''_{\ell}(k,3)$ to be the nilpotence index of $E_{\ell}'(k, 3)$.
We suppose that $\gcd(k, 6) = 1$, that $k = \sum a_i\cdot 3^i$, and that $s \equiv k\pmod{54}$ with $1\leq s\leq 53$.  For $\ell \in \{7, 11\}$ and for all $t\geq 1$, we let $z_{\ell, t} = S_{\ell}^{\prime\prime\prime}(2\cdot 3^t + 1, p) - 1$. For $k\leq 4\cdot 10^4$ and $\ell = 7$, our computations show that 
\[
S_{\ell}^{\prime\prime\prime}(k, p) = S_{\ell}^{\prime\prime\prime}(s, p) + \underset{i \geq 3}\sum w_{\ell, i}\, z_{\ell, i}, 
\]
where the $w_{\ell, i}$ are linear combinations of the coefficients $a_i$ which do not appear to satisfy obvious general formulas.  When $\ell = 11$, the formula holds for all $k\not\equiv 7, 11\pmod{54}$ and for $k$ not of the form $\frac{3^m + 1}{4}$ where $m\geq 1$ is odd.  
Our computations also suggest that 
\[
z_{7, 1} = 1, z_{7, 2} = 2, \ \text{and} \ z_{7, i} = z_{7, i - 1} + 2z_{7, i -2} \ \text{for $i \geq 3$}; 
\]
\[
z_{11, 2} = 2, \ \text{and} \ z_{11, i} =  2z_{11, i - 1} \ \text{for $i \geq 3$};
\]
that for $\ell \in \{7, 11\}$, we have $S_{\ell}^{\prime\prime\prime}(k, p) = O(k^{\alpha})$, where $\alpha = \log_3 2 \approx 0.631$; and that $N_{\ell}(\widetilde{D_{2}}^k, p) \leq S_{\ell}^{\prime\prime\prime}(k, p).$

\medskip

\subsection{Plan for the Paper}
We structure the rest of the paper as follows.  In Section \ref{background}, we give necessary background, and in Section 3, we prove Theorem \ref{simple_bound}, Proposition \ref{vanishing}, and Theorems \ref{partition_5} and \ref{partition_3}.

Before we proceed, we give a brief overview of the proof of Theorem \ref{simple_bound}.  When $T_{\ell}'$ acts nilpotently mod $\ell$ on a subspace $\mathbb{F}_p[\widetilde{f}]\subseteq \widetilde{M}^{(p)}(\Gamma_0(N))$, it lowers degrees in $\widetilde{f}$.  To prove the theorem, we show that $T_{\ell}'$ lowers degrees by at least two, uniformly for all $\ell \neq p$, when $(p, N)\in\{(3, 1), (5, 1), (7, 1), (3, 4)\}$.  When $N = 1$, we decompose $\widetilde{M}^{0, (p)}(\Gamma_0(1)) = \mathbb{F}_p[\widetilde{\Delta}]$ into subspaces of forms whose expansions in $\mathbb{F}_p\llbracket q\rrbracket$ have support on squares mod $p$, non-squares mod~$p$, and multiples of $p$.  The residue class of $\ell$ modulo $p$ determines how $T_{\ell}'$ maps between these subspaces.  

When $(p, N) = (3, 4)$, with $F(z)$ as in \eqref{defn_G}, we note that $\widetilde{M}^{(3)}(\Gamma_0(4)) = \mathbb{F}_3[\widetilde{F}]$ as in \eqref{level_4_gen}, and that $\widetilde{D}_2(z) = \widetilde{F}(z) - \widetilde{F}^3(z)$ in $\mathbb{F}_3\llbracket q\rrbracket$ as in \eqref{D2_equivalence}.  We adapt arguments of Monsky \cite{monsky} to study how $T_{\ell}'$ maps between the subspace $W_1$ generated by $\widetilde{D}_2^k$ for $k\equiv 1\pmod{6}$ and $\widetilde{D}_2^i\widetilde{F}^3$ for $i\equiv 4\pmod{6}$ and the subspace $W_5$ generated by $\widetilde{D}_2^k$ for $k\equiv 5\pmod{6}$ and $\widetilde{D}_2^i\widetilde{F}^3$ for $i\equiv 2 \pmod{6}$.  In both the $N = 1$ and $N = 4$ settings, the mapping properties of Hecke operators on our subspaces allow us to prove our degree-lowering results.

\medskip

\begin{center}
    \footnotesize{ACKNOWLEDGEMENTS}
\end{center}
We thank Anna Medvedovksy for her helpful comments and discussions.  We also thank the referee, whose suggestions improved the paper's presentation.


\section{Background} \label{background}
For details on modular forms, one may consult \cite{Cohen2017ModularFA}, for example.  When $N$ and $k$ are integers with $N\geq 1$, we denote the space of weight $k$ holomorphic modular forms on $\Gamma_0(N)$ by $M_k(\Gamma_0(N))$, and we denote its subspace of cusp forms by $S_k(\Gamma_0(N))$.
We introduce here some of the forms in these spaces that we require.  We let $z\in \mathbb{H}$, the complex upper half-plane, and we let $q = e^{2\pi iz}$.  For even $k\geq 2$, the level one Eisenstein series in weight $k$ is
\[
E_{k}(z) = 1 - \frac{2k}{B_{k}} \sum_{n = 1}^{\infty}\sum_{d\mid n}d^{k - 1} \, q^{n},
\]
where $B_{k}$ denotes the $k$th Bernoulli number. When $k\geq 4$, we have $E_{k} \in M_{k}(\Gamma_{0}(1))$; when $k = 2$ and $N \geq 2$, we have
\begin{equation}
E_{2, N}(z) = N E_{2}(N z) - E_{2}(z) \in M_{2}(\Gamma_{0}(N)). \label{Eisenstein}
\end{equation}
The Dedekind eta-function is defined by
\begin{equation}
\eta(z) = q^{1/24}\prod_{n=1}^{\infty}(1 - q^n). \label{eta}
\end{equation}
The Delta-function,
\begin{equation}
\Delta(z) = \eta(z)^{24} = q\prod_{n = 1}^{\infty}(1 - q^n) \in S_{12}(\Gamma_0(1)), \label{delta}
\end{equation}
plays an important role in our work. 
When $D$ is a discriminant of a quadratic number field, we let $\chi_D$ denote the corresponding Kronecker character.  Some classical identities relating the eta-function to theta functions are
\begin{align}
&\eta(z) = \underset{n\equiv 1\Mod{6}}{\sum_{n\in \Z}}\chi_{12}(n)q^{\frac{n^2}{24}}, \label{eta_expansion}\\
&\eta(z)^3 = \underset{n \ \text{odd}}{\sum_{n\geq 1}}\chi_{-4}(n) \:n \: q^{\frac{n^2}{8}}, \label{eta3_expansion} \\
&\Theta(z) = \frac{\eta(2z)^{5}}{\eta(z)^2 \: \eta(4z)^2} = \sum_{n\in \Z}q^{n^2}. \label{theta}
\end{align}
We note that $\Theta(z)$ is a modular form of weight $1/2$ and level $4$ in Shimura's theory.  

When $p\geq 5$ is prime, we recall some facts in the $N = 1$ setting following \cite{serre} and \cite{swd}.  Since $\widetilde{E}_{p - 1}= 1$, we have $\widetilde{M}_k^{(p)}(\Gamma_0(1))\subseteq 
\widetilde{M}_{k + p - 1}^{(p)}(\Gamma_0(1))$ for all even $k\geq 0$, and we define
\[
\widetilde{M}^{\alpha, (p)}(\Gamma_{0}(1)) = \bigcup\limits_{k \equiv \alpha \Mod{p-1}} \:  \widetilde{M}_{k}^{(p)}(\Gamma_{0}(1)).
\]
Lemma 5(ii) of \cite{swd} implies that
\begin{equation}
\theta = q\frac{d}{dq} \, : \,\widetilde{M}_{k}^{(p)}(\Gamma_0(1)) \longrightarrow \widetilde{M}_{k + p + 1}^{(p)}(\Gamma_0(1)). \label{theta_operator}
\end{equation}

We next discuss Hecke operators.  We let $\ell\nmid N$ be prime, and we let $f(z) = \sum c(n)q^n \in M_k(\Gamma_0(N))$.  We define Hecke operators $T_{\ell}$ on $M_k(\Gamma_0(N))$ by
\begin{equation}\label{Hecke_defn}
f\mid T_{\ell} = \sum_{n} (c(\ell n) + \ell^{k - 1} c(n/\ell)) \: q^{n},
\end{equation}
where $c(n/\ell) = 0$ if $\ell \nmid n$.  The Hecke operator $T_{\ell}$ maps $M_k(\Gamma_0(N))$ to itself.  We let $r\geq 0$, and for all $n\geq 0$, we define $c_r(n)$ by 
\begin{equation}\label{iteration}
f(z)\mid T_{\ell}^r = \sum_{n} c_r(n)q^n,
\end{equation}
where $T_{\ell}^r$ denotes the $r$th iteration of $T_{\ell}$.  For all integers $m$, we let $v_{\ell}(m)$ denote the $\ell$-adic valuation of $m$.  For our applications, we require formulas for $c_r(n)$.  The following lemma gives such formulas depending on $v_{\ell}(n)$.  

\begin{lemma} \label{iterated_coeff}
Let $\ell\nmid N$ be prime, let $r, s\geq 0$; for all $n\geq 0$, let $c_r(n)$ be as in \eqref{iteration}; and suppose that $v_{\ell}(n) = s$.  Then we have 
\begin{align*}
c_r(n) = 
\sum_{i = 0}^{s} \binom{r}{i}\ell^{i(k - 1)}c(\ell^{r - 2i}n)
+ \sum_{j = s + 1}^{\lfloor{\frac{r + s}{2}\rfloor}}\left(\binom{r}{j} - \binom{r}{j - s - 1}
\right)\ell^{j(k - 1)}c(\ell^{r - 2j}n),
\end{align*}
where the second sum is zero when it is undefined (i.e., when $s + 1 > \lfloor\frac{r + s}{2}\rfloor$).  
\end{lemma}
\begin{proof}
We induct on $r$.  The base case follows directly from \eqref{Hecke_defn}.  We let $t\geq 1$, and we suppose that the statement holds for $0\leq r\leq t$.  We fix $s\geq 0$. For $n\geq 0$ with $v_{\ell}(n) = s\geq 0$, using \eqref{Hecke_defn} and the induction hypothesis, we write $c_{t + 1}(n) = c_t(\ell n) +\ell^{k - 1}c_{t}\left(\frac{n}{\ell}\right)$ in terms of the coefficients~$c(n)$.  Since $v_{\ell}(\ell n) = s + 1$ and $v_{\ell}\left(\frac{n}{\ell}\right) = s -1$, the induction hypothesis gives
\begin{align*}
& c_t(\ell n) = \sum_{i = 0}^{s+1}\binom{t}{i}\ell^{i(k - 1)}c(\ell^{t+1 - 2i}n)
+ \sum_{j = s+2}^{\lfloor\frac{t+s+1}{2}\rfloor}\left(\binom{t}{j} - \binom{t}{j - s -2}\right)\ell^{j(k - 1)}c(\ell^{t +1 -2j}n), \\
& c_{t}\left(\frac{n}{\ell}\right) 
= \sum_{i = 0}^{s - 1}\binom{t}{i}\ell^{i(k - 1)}c(\ell^{t - 1 -2i}n)
+ \sum_{j = s}^{\lfloor\frac{t+s-1}{2}\rfloor}\left(\binom{t}{j} - \binom{t}{j - s}\right)\ell^{j(k - 1)}c(\ell^{t - 1 - 2j}n).
\end{align*}
It follows that 
\begin{align*}
c_{t + 1}(n) & = c_{t}(\ell n) + \ell^{k - 1}c_t\left(\frac{n}{\ell}\right) \\
& = \sum_{i = 0}^{s+1}\binom{t}{i}\ell^{i(k - 1)}c(\ell^{t+1 - 2i}n)
+ \sum_{j = s+2}^{\lfloor\frac{t+s+1}{2}\rfloor}\left(\binom{t}{j} - \binom{t}{j - s -2}\right)\ell^{j(k - 1)}c(\ell^{t +1 -2j}n) \\
& + \sum_{i = 0}^{s - 1}\binom{t}{i}\ell^{(i + 1)(k - 1)} c(\ell^{t - 1 - 2i}n)
+ \sum_{j = s}^{\lfloor\frac{t+s-1}{2}\rfloor}\left(\binom{t}{j} - \binom{t}{j - s}\right)\ell^{(j +1)(k - 1)} c(\ell^{t - 1 - 2j}n) \\
& = c(\ell^{t + 1}n) + \sum_{i = 1}^{s}\left(\binom{t}{i} + \binom{t}{i - 1}\right)\ell^{i(k - 1)}c(\ell^{t + 1 - 2i}n) \\
& + \left(\binom{t}{s + 1} + \binom{t}{s} - \binom{t}{0}\right)\ell^{(s + 1)(k - 1)}c(\ell^{t + 1 - 2(s + 1)}n) \\
& + \sum_{j = s + 2}^{\lfloor\frac{t + s + 1}{2}\rfloor}
\left(\left(\binom{t}{j} + \binom{t}{j - 1}\right) - \left(\binom{t}{j - s - 1} + \binom{t}{j - s - 2}\right)\right)\ell^{j(k - 1)}c(\ell^{t + 1 - 2j}n) \\
& = \sum_{i = 0}^{s}\binom{t + 1}{i}\ell^{i(k - 1)}c(\ell^{t + 1 - 2i}n)
+ \sum_{j = s+1}^{\lfloor{\frac{t + s + 1}{2}\rfloor}}
\left(\binom{t + 1}{j} - \binom{t + 1}{j - s - 1}\right)\ell^{j(k - 1)}c(\ell^{r - 2i}n)
\end{align*}
where the third equality follows from shifting $j$ to $j - 1$ in the second sum and combining terms, and the final equality uses Pascal's identity.  
\end{proof}
\noindent 
We have a particular interest in the following special case of the lemma.  
\begin{corollary}\label{coeff_corollary}
Assume the notation in Lemma \ref{iterated_coeff}.  Let $r\geq 1$, let $\ell$ and $p$ be distinct primes with $p\ell\nmid N$, and suppose that $f(z) = \sum c(n)q^n \in M_{k}(\Gamma_{0}(N))_{(p)}$.  Then for all $n\geq 0$ with $\ell\nmid n$, we have 
\begin{equation*}
c_{p^r}(n) \equiv c(\ell^{p^r}n) - \ell^{k-1}c(\ell^{p^r - 2}n) \pmod{p}.
\end{equation*}
\end{corollary} 

\begin{proof}
Letting $s = 0$ and replacing $r$ by $p^r$ in Lemma \ref{iterated_coeff}, it follows for all $n\geq 0$ with $\ell\nmid n$ that 
\begin{align*}
c_{p^r}(n) & = c(\ell^{p^r}n) + \sum_{j = 1}^{\frac{p^r - 1}{2}}\left(\binom{p^r}{j} - \binom{p^r}{j - 1}\right)\ell^{j(k - 1)}c(\ell^{p^{r}-2j}n) \\
& \equiv c(\ell^{p^r}n) - \ell^{k - 1}c(\ell^{p^r - 2}n) \pmod{p},
\end{align*}
where the congruence follows since $\binom{p^r}{m}\equiv 0\Mod{p}$ for $1\leq m \leq p^r - 1$ and 
$\binom{p^r}{0} = 1$.
\end{proof}
\noindent
Lastly, for $m\geq 1$, we recall the $U_m$-operator on $f(z) = \sum c(n)q^n \in M_k(\Gamma_0(N))$.  We have 
\begin{equation}
f(z)\mid U_{m} = \sum c(mn)q^n \label{U_operator}
\end{equation}
and 
\begin{equation*}
U_m : M_k(\Gamma_0(N)) \longrightarrow \begin{cases}
M_k(\Gamma_0(N)), & m\mid N, \\
M_k(\Gamma_0(Nm)), & m\nmid N.
\end{cases}
\end{equation*}


\section{Proofs of Theorem \ref{simple_bound}, Proposition \ref{vanishing}, and Theorems \ref{partition_5} and \ref{partition_3}}

\subsection{Proof of parts (1) and (2) of Theorem \ref{simple_bound}} 
Since the proofs of parts (1) and (2) of Theorem \ref{simple_bound} are similar for $p\in \{3, 5, 7\}$, it suffices to give details for the $p = 5$ case.  To begin, we require a lemma.  We recall from part (1) of Theorem \ref{nilpotent_theorem} that when $\ell \equiv \pm 1 \Mod{5}$, the Hecke operator $T_{\ell}^{'}$ acts locally nilpotently on $\widetilde{M}^{0, (5)}(\Gamma_{0}(1)) = \mathbb{F}_{5}[\widetilde{\Delta}]$.
\begin{lemma}
The space $\widetilde{M}^{0, (5)}(\Gamma_{0}(1))$ has a basis $B =  \{\widetilde{f}_{5i + j} : i \geq 0 \:, \: 0 \leq j \leq 4\}$, where
\begin{gather*}
\widetilde{f}_{5i + j} = q^{5i+j} + \ldots = 
\begin{cases}
    \displaystyle{\sum_{n \equiv 1,4 \Mod{5}} \widetilde{a_{i,j}(n)}q^{n}, \quad j \in \{1,4\}}, \\
    \displaystyle{\sum_{n \equiv 2,3 \Mod{5}} \widetilde{a_{i,j}(n)} q^{n}, \quad j \in \{2,3\}}. \\
\end{cases}
\end{gather*}
\end{lemma}
\begin{proof}
We define
\begin{align*}
\widetilde{f}_{0} &= 1, \\
\widetilde{f}_{1} &= \widetilde{\Delta} + 4 \widetilde{\Delta}^{2} =  q + \ldots, \\
\widetilde{f}_{2} &= \widetilde{\Delta}^{2} = q^{2} + \ldots,\\
\widetilde{f}_{3} &= \widetilde{\Delta}^{3} + 2 \widetilde{\Delta}^{4} + 3 \widetilde{\Delta}^{5} = q^{3} + \ldots, \\
\widetilde{f}_{4} &= \widetilde{\Delta}^{4} + \widetilde{\Delta}^{5} = q^{4} + \ldots. \\
\end{align*}
For all $i \geq 0$ and $0 \leq j \leq 4$, \text{we let} $\widetilde{f}_{5i + j} = \widetilde{\Delta}^{5i} \widetilde{f}_{j} =  q^{5i + j} + \ldots$. With $\theta = q\frac{d}{dq}$ as in \eqref{theta_operator}, we claim that 
\[
\theta^{2}(\widetilde{f}_{5i + j}) = \left( \frac{j}{5} \right) \widetilde{f}_{5i + j},
\]
where $\left(\frac{\bigcdot}{5}\right)$ is the Legendre symbol.  We compute
$\theta^{2}(\widetilde{f}_{j}) = \widetilde{f}_{j} \:\:\text{for}\:\: j \in \{1,4\} \:\:\text{and}\:\: \theta^{2}(\widetilde{f}_{j}) = - \widetilde{f}_{j} \:\:\text{for}\:\: j \in \{2,3\}.$ 
Furthermore, for all $i \geq 0$ and $1 \leq j \leq 4$, we have
\begin{align*}
 \theta^{2}(\widetilde{f}_{5i+j}) & = \theta(\widetilde{\Delta}^{5i} \theta(\widetilde{f}_{j})) = \widetilde{\Delta}^{5i} \: \theta^{2}(\widetilde{f}_{j}) + \theta(\widetilde{f}_{j}) \: \theta(\widetilde{\Delta}^{5i}) \\ 
 &= \widetilde{\Delta}^{5i} \theta^{2}(\widetilde{f}_{j}) =
\begin{cases}
\widetilde{f}_{5i + j}, \quad j \in \{1,4\}, \\
- \widetilde{f}_{5i + j}, \quad j \in \{2,3\}.
\end{cases}
\end{align*}
To conclude, we observe that $\theta^{2}(\widetilde{f}_{5i+j}) = \widetilde{f}_{5i+j}$ for $j \in \{1,4\}$ implies that $\widetilde{f}_{5i + j}$ is supported on squares modulo $5$, and for $j \in \{2,3\}$ that $\theta^{2}(\widetilde{f}_{5i+j}) = - \widetilde{f}_{5i+j}$ implies that $ \widetilde{f}_{5i + j}$ is supported on non-squares modulo 5.
\end{proof}
\noindent
We also require explicit representations of powers of $\widetilde{\Delta}$ in $\F_{5}[[q]]$ in the basis $B$:
\begin{gather}\label{del_B}
   \widetilde{\Delta}^{5i + j} \equiv 
   \begin{cases}
   \widetilde{f}_{5i} &\text{if $j = 0$},\quad \\
       \widetilde{f}_{5i + 1} + \widetilde{f}_{5i + 2} &\text{if $j = 1$},\quad \\
			\widetilde{f}_{5i + 2} \quad &\text{if $j = 2$}, \quad \\
    \widetilde{f}_{5i + 3} + 3 \widetilde{f}_{5i + 4} + 4 \widetilde{f}_{5(i+1)} &\text{if $j = 3$},\quad \\
			\widetilde{f}_{5i + 4} + 4 \widetilde{f}_{5(i+1)} \quad &\text{if $j = 4$}. \quad
   \end{cases}
\end{gather} 
We denote the subspace of $\widetilde{M}^{0, (5)}(\Gamma_{0}(1))$ of cusp forms by $\widetilde{S}^{0, (5)}(\Gamma_{0}(1))$, and we note that $\widetilde{S}^{0, (5)}(\Gamma_{0}(1)) = \widetilde{\Delta} \mathbb{F}_{5}[\widetilde{\Delta}]$. We consider the following subspaces of $\widetilde{S}^{0} = \widetilde{S}^{0, (5)}(\Gamma_{0}(1))$:
\begin{align*}
 \widetilde{S}^{0, 1} &= \{\widetilde{f} \in \widetilde{S}^{0} \: | \: \theta^{2}(\widetilde{f}) = \widetilde{f} \} = \text{Span}_{\mathbb{\F}_{5}}\{\widetilde{f}_{5i + j} \: | \: j \in \{1,4\} \},  \\
 \widetilde{S}^{0, - 1} &= \{f \in \widetilde{S}^{0} \: | \: \theta^{2}(\widetilde{f}) = - \widetilde{f} \} = \text{Span}_{\mathbb{F}_{5}}\{\widetilde{f}_{5i + j} \: | \: j \in \{2,3\} \},  \\ 
 \widetilde{S}^{0, 0} &= \{\widetilde{f} \in \widetilde{S}^{0} \: | \: \theta^{2}(\widetilde{f}) = 0 \} = \text{Span}_{\mathbb{F}_{5}}\{\widetilde{f}_{5i} \: | \: i \geq 1 \}.
\end{align*}
Let $\ell \neq 5$ be prime. It follows from \eqref{Hecke_defn} that
\begin{align}
T_{\ell}^{'} &: \widetilde{S}^{0, j} \to \widetilde{S}^{0, j} \ \text{for} \ j \in \{-1,0,1\} \:\: \text{and} \:\: \ell \equiv \pm 1 \Mod{5}; \label{pm_1_map} \\
T_{\ell}^{'} &: \widetilde{S}^{0, 1} \to \widetilde{S}^{0, -1}, \:\: \widetilde{S}^{0, -1} \to \widetilde{S}^{0, 1}, \:\: \widetilde{S}^{0, 0} \to \widetilde{S}^{0, 0} \ \text{for} \ \ell \equiv \pm 2 \Mod{5}. \notag
\end{align}
We also observe that $\widetilde{S}^{0} = \widetilde{S}^{0,1} \bigoplus \widetilde{S}^{0,-1} \bigoplus \widetilde{S}^{0,0}$.
Let $\ell \equiv \pm{1} \Mod{5}$ be prime, let $i \geq 0$, and let $0 \leq j \leq 4$.
We prove part (2) of Theorem \ref{simple_bound} for $p = 5$ in the following form:
\begin{gather}\label{bound_5}
N_{\ell}(\widetilde{\Delta}^{5 i + j}, 5)  \leq 
\begin{cases}
    2 i + 1, \quad j \in \{1,2\}, \\
    2 i + 2, \quad j \in \{3,4\}, \\
    i, \quad \quad \quad j = 0.
\end{cases}
\end{gather}
\noindent
We first claim that in order to prove \eqref{bound_5}, it suffices to show that    
\begin{gather}\label{index_5}
N_{\ell}(\widetilde{f}_{5 i + j}, 5)  \leq 
\begin{cases}
    2 i + 1, \quad j \in \{1,2\}, \\
    2 i + 2, \quad j \in \{3,4\}, \\
    i, \quad \quad \quad j = 0.
\end{cases}
\end{gather}
To see this, we let $\ell \equiv \pm 1 \Mod{5}$ and $j = 3$, and we assume that \eqref{index_5} holds.  For all $i \geq 0$, we use \eqref{del_B} to obtain $\widetilde{\Delta}^{5i + 3} \mid T_{\ell}^{'} = (\widetilde{f}_{5i + 3} + 3 \: \widetilde{f}_{5i + 4} + 4 \: \widetilde{f}_{5(i+1)}) \mid T_{\ell}^{'}$ in $\F_{5}[[q]]$. Using \eqref{index_5}, it follows that
$N_{\ell}(\widetilde{\Delta}^{5i+3}, 5) \leq \max\{N_{\ell}(\widetilde{f}_{5 i + 3}, 5), N_{\ell}(\widetilde{f}_{5 i + 4}, 5), N_{\ell}(\widetilde{f}_{5 (i + 1)}, 5) \} \leq 2 i + 2$. 
The arguments for $j \neq 3$ are similar. 

\medskip

We now prove \eqref{index_5}.
We will show for all $i \geq 1$, that 
\begin{gather}\label{degree_low_5}
N_{\ell}(\widetilde{f}_{5i + j}, 5) \leq 1 + 
\begin{cases}
 \max \{ N_{\ell}(\widetilde{f}_{k}, 5) : 1 \leq k \leq 5(i - 1), \: 5 \mid k\}, \quad \quad \: \: j = 0, \\
\max \{ N_{\ell}(\widetilde{f}_{k}, 5) : 1 \leq k \leq 5i - 1, \: \left( \frac{k}{5} \right) = 1 \}, \: \: \: \quad j = 1, \\
\max \{ N_{\ell}(\widetilde{f}_{k}, 5) : 1 \leq k \leq 5i - 2, \: \left( \frac{k}{5} \right) = -1 \}, \quad j = 2, \\
\max \{ N_{\ell}(\widetilde{f}_{k}, 5) : 1 \leq k \leq 5i + 2, \: \left( \frac{k}{5} \right) = -1 \}, \quad j = 3, \\
\max \{ N_{\ell}(\widetilde{f}_{k}, 5) : 1 \leq k \leq 5i + 1, \: \left( \frac{k}{5} \right) = 1\}, \:\: \: \quad j = 4. \\
\end{cases}
\end{gather}
To illustrate, we prove \eqref{degree_low_5} for $j = 1$. 
We observe that
\[
\widetilde{f}_{5i + 1} \in \widetilde{S}_{12(5i + 2)}^{0, (5)} \: \medcap\: \widetilde{S}^{0,1} = \text{Span}_{\mathbb{F}_{5}}\{\widetilde{f}_{5a + b} \: | \: a \geq 0 \:, \: b \in \{1,4\} \:, \:5a + b \leq 5i + 1 \}.
\]
When $\ell \equiv \pm{1} \Mod{5}$, we see from \eqref{pm_1_map} that $T_{\ell}^{'}$ maps this space to itself.  It follows that
\begin{equation*}
\widetilde{f}_{5i + 1} \mid T_{\ell}^{'} = \sum_{r = 0}^{i} \widetilde{c}_{r} \widetilde{f}_{5r + 1} + \sum_{s = 0}^{i - 1} \widetilde{d}_{s} \widetilde{f}_{5s + 4}.
\end{equation*}
We claim that $\widetilde{c_{i}} = 0$ in $\F_{5}$. We suppose that $\widetilde{c}_{i} \neq 0 $.  We compute
\begin{align*}
\textup{deg}_{\widetilde{\Delta}}(\widetilde{f}_{5i + 1}) = 5i + 2 > 5i &=  \underset{0 \leq r, s \leq i - 1}{\max} \{ \textup{deg}_{\widetilde{\Delta}}(\widetilde{f}_{5r + 1}) \: , \: \textup{deg}_{\widetilde{\Delta}}(\widetilde{f}_{5s + 4})\}\\ 
&= \max_{0 \leq r, s \leq i - 1} \{ 5r + 2 \: , \: 5(s + 1)\} 
= \textup{deg}_{\widetilde{\Delta}} (\widetilde{f}_{5i - 1}).
\end{align*}
Therefore, $\widetilde{c}_{i} \neq 0$ implies that $\textup{deg}_{\widetilde{\Delta}}(\widetilde{f}_{5i + 1}\mid T_{\ell}^{'}) = 5i + 2 = \textup{deg}_{\widetilde{\Delta}} (\widetilde{f}_{5i + 1})$, which contradicts nilpotency of $T_{\ell}^{\prime}$.  We conclude that
\begin{equation}
\widetilde{f}_{5i + 1} \mid T_{\ell}^{'} = \sum_{r = 0}^{i - 1} \widetilde{c}_{r} \widetilde{f}_{5r + 1} + \sum_{s = 0}^{i - 1} \widetilde{d}_{s} \widetilde{f}_{5s + 4}\label{sum}
\end{equation}
which gives \eqref{degree_low_5} for $j = 1.$ The proof of the remaining cases follow similarly.

\medskip

\noindent
We now use induction to prove \eqref{index_5}, and with it, the $p = 5$ case of part (2) of Theorem \ref{simple_bound}.
\vspace{1mm}
\noindent
The base cases are 
\[
N_{\ell}(\widetilde{f}_{1}, 5) = N_{\ell}(\widetilde{f}_{2}, 5) = 1, \ N_{\ell}(\widetilde{f}_{3}, 5)  \leq 1 + N_{\ell}(\widetilde{f}_{2}, 5) = 2, \ N_{\ell}(\widetilde{f}_{4}, 5) \leq 1 + N_{\ell}(\widetilde{f}_{2}, 5) = 2.
\]
We let $i \geq 1$ and $0 \leq b \leq 4$, and we suppose for all $0 \leq a \leq i - 1$ that the statement holds for $5a + b$. To illustrate the general argument, we prove the statement for $5i +j$ when $j = 1$.  Using the induction hypothesis \eqref{index_5} and \eqref{sum}, we find that
\begin{align*}
N_{\ell}(\widetilde{f}_{5i + 1}, 5) &\leq 1 + \max\{N_{\ell}(\widetilde{f}_{k}, 5) : 1 \leq k \leq 5(i - 1) + 4, \, k\equiv 1, 4 \Mod{5}\} \\
 &\leq 1 + \max \{N_{\ell}(\widetilde{f}_{1}, 5), \: N_{\ell}(\widetilde{f}_{4}, 5), \: \ldots, \: N_{\ell}(\widetilde{f}_{5(i - 1) + 1}, 5), \: N_{\ell}(\widetilde{f}_{5(i - 1) + 4}, 5) \} \\
 &\leq 1 + \max \{1, \: 2, \: \ldots, 2(i - 1) + 1, \: 2(i - 1) + 2\} = 2i + 1,
 \end{align*}
 which proves \eqref{index_5} for $j = 1$.  


\vspace{2mm}
\subsection{Proof of Part (3) of Theorem \ref{simple_bound}}
We shift our attention to $\widetilde{M}^{(3)}(\Gamma_{0}(4))$. We let 
\begin{align}
& P(z) = E_{2,2}(z) = 2 E_{2}(2z) - E_{2}(z) \in M_{2}(\Gamma_{0}(2)), \ 
A(z) = \frac{\eta(2z)^{16}}{\eta(z)^{8}} \in M_{4}(\Gamma_{0}(2)), \notag \\
& F(z) = \frac{\eta(4z)^{8}}{\eta(2z)^{4}} \in M_{2}(\Gamma_{0}(4)), 
\ 
G(z) = \Delta(2z) \in S_{12}(\Gamma_{0}(2)), \label{defn_G}
\end{align}
where $E_{2,2}(z)$ is an Eisenstein series as in \eqref{Eisenstein}. 
\noindent
We recall from Section 1 that $D_2(z) = \eta(2z)^{12} \in S_{6}(\Gamma_{0}(4))$, and we note that 
\begin{align}
& D_2(z)^2  = G(z), \label{D2G} \\
& \widetilde{D}_{2}(z) = \widetilde{F}(z) - \widetilde{F}^{3}(z) \: \text{in} \: \F_{3}[[q]], \label{D2_equivalence} \\
& \widetilde{P}(z) = 1 \: \text{in} \: \F_{3}[[q]]. \label{P_cong}
\end{align}
\noindent
We use the next theorem to prove bounds on $N_{\ell}(\widetilde{D}_{2}^{k}, 3)$ for primes $\ell \not\in \{2, 3\}$ and $k \geq 1$. Its proof is an adaptation of Monsky's work \cite{monsky} in $\widetilde{M}^{(2)}(\Gamma_{0}(3))$.
\begin{theorem} \label{N4_l3}
Consider the following subspaces of $\widetilde{M}^{(3)}(\Gamma_{0}(4))$:
\begin{align*}
W_{1} = {\rm Span}_{\mathbb{F}_3}\{ \widetilde{D}_{2}^{k}, \: \widetilde{D}_{2}^{i} \widetilde{F}^{3} : k \equiv 1 \Mod{6} \: \text{and} \: i \equiv 4 \Mod{6} \},  \\
W_{5} = {\rm Span}_{\mathbb{F}_3}\{ \widetilde{D}_{2}^{k}, \: \widetilde{D}_{2}^{i}\widetilde{F}^{3} : k \equiv 5 \Mod{6} \: \text{and} \: i \equiv 2 \Mod{6} \}.
\end{align*}
We have
\begin{gather*}
\displaystyle{T_{\ell}^{'} : \,
\begin{dcases}
			W_{1} \to W_{1}, \: W_{5} \to W_{5} &\text{if $\ell \equiv 1 \Mod{6}$},\quad \\
			 W_{1} \to W_{5}, \: W_{5} \to W_{1} &\text{if $\ell \equiv -1 \Mod{6}.$} \quad
\end{dcases} 
  }
\end{gather*}
\end{theorem}

Our proof of Theorem \ref{N4_l3} requires us to study forms on $\Gamma_{0}(2)$.  We recall, for all even $k \geq 0$, that $\textup{dim}(M_{k}(\Gamma_{0}(2))) = 1 + \Big\lfloor \frac{k}{4} \Big\rfloor$. Thus, for all $j \geq 1$, we have $\textup{dim}(\widetilde{M}_{8j}^{(3)}(\Gamma_{0}(2))) \leq 2j + 1$.
Since 
\begin{equation}
    B_j = \{\widetilde{A}^{i} \: | \: 0 \leq i \leq 2j\} \label{B_j}
\end{equation} 
is linearly independent, it is a basis for $\widetilde{M}_{8j}^{(3)}(\Gamma_{0}(2))$.  For all even $t \geq 0$, we define
\begin{align*}
\widetilde{K}_{t}^{(3)}(\Gamma_{0}(2)) &= \textup{Ker}(U_3) \:\: \text{on} \:\: \widetilde{M}_{t}^{(3)}(\Gamma_{0}(2)),  \\
\widetilde{K}^{(3)}(\Gamma_{0}(2)) &= \bigcup_{t \geq 0,\: t \: even} \widetilde{K}_{t}^{(3)}(\Gamma_{0}(2)) \subseteq \widetilde{M}^{(3)}(\Gamma_{0}(2)),
\end{align*}
where $U_3$ is the $U$-operator as in \eqref{U_operator}.
We investigate properties of $\widetilde{K}^{(3)}(\Gamma_{0}(2))$ in the following lemma.
\begin{lemma} \label{basis_lemma}
We assume the notation above.  For all $i\geq 0$, let 
\begin{equation}
    \widetilde{g}_i(z) = \widetilde{A}^{i}(\widetilde{A} - \widetilde{A}^2). \label{g_i}
\end{equation}
 A basis for $\widetilde{K}^{(3)}(\Gamma_{0}(2))$ is 
 \[
 B^{'} = \{\widetilde{g}_i : i \geq 0 \:\: \text{and} \:\: i \not\equiv 2 \Mod{3} \}.
 \]
\end{lemma}
\begin{proof}
Since $P(z)$ has weight 2 and using $\widetilde{P}(z) = 1$ as in \eqref{P_cong}, it suffices to show, for all $j \geq 1$, that
\begin{equation*}
B_{j}^{'} = \{\widetilde{g}_i(z) = q^{i + 1} + \ldots : 0 \leq i \leq 2j - 2 \:\: \text{and} \:\: i \not\equiv 2 \Mod{3}\}
\end{equation*}
is a basis for $\widetilde{K}_{8j}^{(3)}(\Gamma_{0}(2))$.
We note that $B_{j}^{'}$ is linearly independent. To show that it spans, we let $f(z) \in \widetilde{K}_{8j}^{(3)}(\Gamma_{0}(2))$. Since $B_{j}^{'} \subseteq \widetilde{K}_{8j}^{(3)}(\Gamma_{0}(2))$, for all $0 \leq i \leq 2j - 2$ with $i \neq 2 \Mod{3}$, there exists $\widetilde{\alpha}_{i} \in \mathbb{F}_{3}$ such that
\[
\widetilde{f}(z) - \sum_{i = 0}^{2j - 2} \widetilde{\alpha}_{i} \: \widetilde{g}_i(z) = \widetilde{c} q^{2j} + \ldots \in \widetilde{K}_{8j}^{(3)}(\Gamma_{0}(2)) \subseteq \widetilde{M}_{8j}^{(3)}(\Gamma_{0}(2))
\]
for some $\widetilde{c}\in \mathbb{F}_3$.
Since $\widetilde{M}_{8j}^{(3)}(\Gamma_{0}(2))$ has basis $B_{j}$ in \eqref{B_j}, we must have 
\[
\widetilde{f} - \sum_{i = 0}^{2j - 2} \widetilde{\alpha}_{i} \: \widetilde{g}_i(z) = \widetilde{c} \widetilde{A}^{2j}.
\]
It follows that $\widetilde{c} \widetilde{A}^{2j} \in \widetilde{K}_{8j}^{(3)}(\Gamma_{0}(2))$. However, for all $t \geq 0$, the coefficient of $q^{3t + 3}$ in $\widetilde{A}^{3t}, \widetilde{A}^{3t + 1}$, and $\widetilde{A}^{3t + 2}$ is nonzero in $\F_{3}[[q]]$. We conclude that $\widetilde{c} = 0$, and therefore, that $\widetilde{M}_{8j}^{(3)}(\Gamma_{0}(2))\subseteq \text{Span}_{\mathbb{F}_3}(B_j^{\prime})$.
\end{proof}
\begin{lemma}
With $G(z)$ as in \eqref{defn_G}, we have $\widetilde{K}^{(3)}(\Gamma_{0}(2)) \subseteq \mathbb{F}_{3}[\widetilde{\Delta}, \widetilde{G}].$
\end{lemma} 
\begin{proof}
In view of Lemma \ref{basis_lemma}, it suffices to show, for all $i\not\equiv 2\Mod{3}$, that $\widetilde{g}_{i}$ as in \eqref{g_i} lies in $\mathbb{F}_{3}[\widetilde{\Delta},\widetilde{G}]$.
We note that
\begin{align*}
    \widetilde{g}_0 &= \widetilde{A} - \widetilde{A}^{2} = \widetilde{\Delta} + \widetilde{G}, \\
    \widetilde{g}_1 &= \widetilde{A}(\widetilde{A} - \widetilde{A}^{2}) = \widetilde{G}, \\ 
    \widetilde{g}_3 &= \widetilde{A}^{3} (\widetilde{A} - \widetilde{A}^{2}) =  - \widetilde{\Delta}^{2} + \widetilde{G} + \widetilde{G}^{2}, \\
    \widetilde{g}_4 &= \widetilde{A}^{4}(\widetilde{A} - \widetilde{A}^{2}) =  - \widetilde{\Delta}^{2} + \widetilde{G}. \\
\end{align*}
Let $n \geq 1$. We have
\begin{align*}
\widetilde{g}_{3n} - \widetilde{g}_{3(n + 1)} &= \widetilde{A}^{3n} (\widetilde{A} - \widetilde{A}^{2}) - \widetilde{A}^{3(n+1)}(\widetilde{A} - \widetilde{A}^{2}) 
= (\widetilde{A}^{3n} - \widetilde{A}^{3n+1}) \: \widetilde{g}_{0} \\
& = (\widetilde{A}^{n} - \widetilde{A}^{n+1})^{3} \: \widetilde{g}_{0} 
= \widetilde{A}^{3(n-1)} (\widetilde{A} - \widetilde{A}^{2}) \: \widetilde{g}_{0}^{3}
= \widetilde{g}_{3(n-1)} \: \widetilde{g}_{0}^{3}.
\end{align*} 
Similarly, we find that $\widetilde{g}_{3n + 1} - \widetilde{g}_{3(n+1) + 1} = \widetilde{g}_{3(n-1)+1} \: \widetilde{g}_{0}^{3}$.
Therefore, for all $i \geq 5$ with $i \not\equiv 2 \Mod{3}$, we have $\widetilde{g}_{i} = \widetilde{g}_{i-3} - \widetilde{g}_{i - 6} \: \widetilde{g}_{0}^{3}$. The claim follows by induction on $i$.
\end{proof}
\begin{lemma}
The space $\widetilde{K}^{(3)}(\Gamma_{0}(2))$ is a free $\mathbb{F}_{3}[\widetilde{G}^{3}]$-module with basis 
\[
\{\widetilde{\Delta}, \widetilde{\Delta}^{2}, \widetilde{\Delta} \widetilde{G}^{2}, \widetilde{\Delta}^{2} \widetilde{G}, \widetilde{G}, \widetilde{G}^{2}\}.
\]
\end{lemma}
\begin{proof}
Since $\widetilde{G}(z)^{3}\in q^{3} \: \mathbb{F}_{3}[q^{3}]$, the space $\widetilde{K}^{(3)}(\Gamma_{0}(2))$ is an $\mathbb{F}_{3}[\widetilde{G}^{3}]$-module. With $h(x) = x^{3} - \widetilde{G}x + \widetilde{G}^{3} \in (\mathbb{F}_{3}[\widetilde{G}])[x]$, we observe that
$h(\widetilde{\Delta}) = \widetilde{\Delta}^{3} - \widetilde{\Delta} \widetilde{G} + \widetilde{G}^{3} = 0$. It follows that $[\mathbb{F}_{3}(\widetilde{\Delta}, \widetilde{G}) : \mathbb{F}_{3}(\widetilde{G})] = 3$. Therefore, $\mathbb{F}_{3}[\widetilde{\Delta}, \widetilde{G}]$ is a free $\mathbb{F}_{3}[\widetilde{G}^{3}]$-module with basis 
\[
\{1, \widetilde{\Delta}, \widetilde{\Delta}^{2}, \widetilde{G}, \widetilde{\Delta} \widetilde{G}, \widetilde{\Delta}^{2}\widetilde{G}, \widetilde{G}^{2}, \widetilde{\Delta} \widetilde{G}^{2}, \widetilde{\Delta}^{2} \widetilde{G}^{2}\}.
\]
We note that $1, \widetilde{\Delta} \widetilde{G}, \widetilde{\Delta}^{2}\widetilde{G}^{2} \in q^{3}\mathbb{F}_{3}[[q^{3}]]$ while the other forms in the basis lie in $\widetilde{K}^{(3)}(\Gamma_{0}(2))$. It follows that $\widetilde{K}^{(3)}(\Gamma_{0}(2))$ is a free $\mathbb{F}_{3}[\widetilde{G}^{3}]$-module with basis $\{\widetilde{\Delta}, \widetilde{\Delta}^{2}, \widetilde{\Delta} \widetilde{G}^{2}, \widetilde{\Delta}^{2}\widetilde{G}, \widetilde{G}, \widetilde{G}^{2}\}$.
\end{proof}
Let $i \in \{1, 5\}$. We consider the map $\rho_{6,i}$ on $\mathbb{F}_{3}[[q]]$ defined by 
\begin{equation} \label{projection}
\rho_{6,i}  \left(  \sum_{n}c(n)\: q^{n} \right) = \sum_{n \equiv i \Mod{6}} c(n)q^{n}.
\end{equation}
We note that $\rho_{6,i}$ is $\mathbb{F}_{3}[\widetilde{G}^{3}]$-linear.  For primes $\ell \notin \{ 2, 3 \}$, the definitions \eqref{Hecke_defn} and \eqref{projection} imply that 
\begin{equation}\label{commutativity}
\rho_{6,i}(\widetilde{f}) \mid T_{\ell} = \rho_{6, \: i \ell^{-1}}(\widetilde{f} \mid T_{\ell}).
\end{equation}
We also note that the $\mathbb{F}_{3}[\widetilde{G}^{3}]$-submodules $V_{1} = \textup{Ker} (\rho_{6,1})$ and $V_{5} = \textup{Ker} (\rho_{6,5})$ have bases \\ $\{\widetilde{\Delta}^{2}, \widetilde{G}, \widetilde{G}^{2}, \widetilde{\Delta} \widetilde{G}^{2}\}$ and $\{ \widetilde{\Delta}, \widetilde{G}, \widetilde{\Delta}^{2} \widetilde{G}, \widetilde{G}^{2} \}$, respectively.
From \eqref{Hecke_defn}, we observe that
\begin{gather}
\displaystyle{T_{\ell}^{'} : 
\begin{dcases}
V_{1} \to V_{1}, \, V_{5} \to V_{5} &\text{if $\ell \equiv 1 \Mod{6}$},\quad \\
V_{1} \to V_{5}, \, V_{5} \to V_{1} &\text{if $\ell \equiv -1 \Mod{6}$}.\quad
\end{dcases} \label{V_map}
  }
\end{gather}
We require
\begin{align}
\rho_{6,1}(\widetilde{\Delta}) = \widetilde{D}_{2},  \:\: \rho_{6,1}(\widetilde{\Delta}^2 \widetilde{G}) = - \widetilde{D}_{2}^{4} \widetilde{F}^{3}, \:\: \rho_{6,1}(\widetilde{G}) =  \rho_{6,1}(\widetilde{G}^2) = 0, \label{rho_1}\\
\rho_{6,5}(\widetilde{\Delta}^2) = - \widetilde{D}_{2}^2 \widetilde{F}^{3},  \:\: \rho_{6,5}(\widetilde{\Delta} \widetilde{G}^{2}) = \widetilde{D}_{2}^{5}, \:\: \rho_{6,5}(\widetilde{G}) =  \rho_{6,5}(\widetilde{G}^2) = 0. \label{rho_5}
\end{align}
We turn to the proof of Theorem \ref{N4_l3}.
\subsubsection{Proof of Theorem \ref{N4_l3}.}
We first observe, using \eqref{D2G} and \eqref{rho_1}, that $\rho_{6,1}$ maps $V_{5}$ onto $W_{1}$ via 
\begin{align}
\rho_{6,1} (\widetilde{\Delta} \cdot \widetilde{G}^{3m}) &= \widetilde{D}_{2}^{6m + 1}, \label{rho_1_1} \\
\rho_{6,1}(\widetilde{\Delta}^{2} \widetilde{G} \cdot G^{3m}) &= - \widetilde{D}_{2}^{6m + 4} \widetilde{F}^{3} \label{rho_1_2}.
\end{align}
Similarly, using \eqref{D2G} and \eqref{rho_5}, we find that $\rho_{6,5}$ maps $V_{1}$ onto $W_{5}$ via
\begin{align}
\rho_{6,5} (\widetilde{\Delta}^{2} \cdot \widetilde{G}^{3m}) &= - \widetilde{D}_{2}^{6m + 2} \widetilde{F}^{3}  \label{rho_5_1}, \\
\rho_{6,5}(\widetilde{\Delta} \widetilde{G}^{2} \cdot \widetilde{G}^{3m}) &= \widetilde{D}_{2}^{6m + 5} \label{rho_5_2}.
\end{align}
Let $\ell \equiv 1 \Mod{6}$, and let $m\geq 0$.  We use \eqref{commutativity}, \eqref{V_map}, \eqref{rho_1_1} and \eqref{rho_1_2} to compute
\begin{align*}
\widetilde{D}_{2}^{6m + 1} \mid T_{\ell}^{'} &= \rho_{6,1} (\widetilde{\Delta} \cdot \widetilde{G}^{3m}) \mid T_{\ell}^{'} \equiv 
\rho_{6,1} ((\widetilde{\Delta}\cdot \widetilde{G}^{3m}) \mid T_{\ell}^{'}) \\
 & = \rho_{6,1} \left( \sum_{r \geq 0} a_{r} \widetilde{\Delta} \cdot \widetilde{G}^{3r} + \sum_{s \geq 0} b_{s} \widetilde{G} \cdot \widetilde{G}^{3s} + \sum_{t \geq 0} c_{t} \widetilde{G}^{2} \cdot \widetilde{G}^{3t} + \sum_{u \geq 0} d_{u} (\widetilde{\Delta}^{2} \widetilde{G}) \cdot \widetilde{G}^{3u} \right) \\
& = \sum_{r \geq 0} a_{r} \widetilde{D}_{2}^{6r + 1} - \sum_{u \geq 0} d_{u} \widetilde{D}_{2}^{6u + 4} \widetilde{F}^{3} \ \text{in} \ \mathbb{F}_3\llbracket q\rrbracket, 
\end{align*}
which gives $\widetilde{D}_2^{6m+1}\mid T_{\ell}^{'}\in W_1$.
We argue similarly using \eqref{commutativity}, \eqref{V_map}, \eqref{rho_1_1}, \eqref{rho_1_2}, \eqref{rho_5_1}, and \eqref{rho_5_2} to complete the proof that $T_{\ell}^{'}$ preserves $W_{1}$ and $W_{5}$ when $\ell \equiv 1 \Mod{6}$ and swaps $W_{1}$ and $W_{5}$ when $\ell\equiv 5\Mod{6}$. 

\subsubsection{Conclusion of Proof of Part (3) of Theorem \ref{simple_bound}}
With $\Theta(z)$ as in \eqref{theta}, we note that Proposition 15.1.2 of \cite{Cohen2017ModularFA} implies that $M(\Gamma_{0}(4)) = \mathbb{C}[\Theta^{4}, F]$.  We also observe that $\widetilde{\Theta}^{4} = 1 + 2\widetilde{F}$ in $\F_{3}[[q]]$.  It follows that 
\begin{equation}
\label{level_4_gen}
\widetilde{M}^{(3)}(\Gamma_{0}(4))= \mathbb{F}_{3}[\widetilde{F}].
\end{equation}
Hence, from \eqref{D2_equivalence} we see that $\widetilde{f} \in W_1$ has $\text{deg}_{\widetilde{F}}(\widetilde{f}) \in \{3, 15\}$ mod $18$, while
$\widetilde{f} \in W_5$ has $\text{deg}_{\widetilde{F}}(\widetilde{f}) \in \{9, 15\}$ mod $18$.  We conclude part (3) of Theorem \ref{simple_bound} on applying Theorem \ref{N4_l3} while observing from part (3) of Theorem \ref{nilpotent_theorem} that $T_{\ell}^{\prime}$ lowers degrees in $\widetilde{F}$ for primes $\ell\not\in \{2, 3\}$.     

\subsection{Proof of Proposition \ref{vanishing}}
We use the theta-series expansion for $\eta(z)$ in \eqref{eta_expansion} to prove part one of the proposition for $(\delta, p) = (1, 2)$.  The proofs for $(\delta, p) \in \{(1, 23), (2, 11), (3, 7), (4, 5)\}$ are similar, and are an adaptation of an argument of Elkies \cite{elkies}.

Let $m$ be odd, and let $B(z) = \eta(8z)\eta(8\cdot 2^mz)$.  We observe that
$\widetilde{\Delta}(z)^{\frac{1 + 2^m}{3}}=\widetilde{B}(z)$ in $\F_{2}[[q]]$.
Using \eqref{eta_expansion}, we obtain
\[
\widetilde{\Delta}(z)^{\frac{1 + 2^m}{3}} = \sum_{\substack{a, b \in \mathbb{Z} \\ a, b \equiv 1 \pmod{6}}} \chi_{12}(ab) \: q^{\frac{a^2 + 2^m b^2}{3}}.
\]
Let $\ell$ be prime with $\left(  \frac{-2}{\ell}  \right) = -1$. We claim that $ \ell \mid (a^2 + 2^m b^2)$ implies that $a, \: b \equiv 0 \Mod{\ell}$.  
We suppose that $\ell \nmid b$. Then we have $(ab^{-1})^{2} \equiv -2^m \Mod{\ell}$. Since $m$ is odd, we find that $\left(\frac{-2^m}{\ell}\right) =  \left(\frac{-2}{\ell} \right) = -1$, which is a contradiction.
Let $\widetilde{\Delta}(z)^{\frac{1 + 2^m}{3}} = \sum \widetilde{a(n)} q^{n}$. Since $\widetilde{\Delta}(z)^{\frac{1 + 2^m}{3}}$ has even weight, \eqref{Hecke_defn} implies that
\begin{equation} \label{del_act_2}
\widetilde{\Delta}(z)^{\frac{1 + 2^m}{3}} \mid T_{\ell} = \sum \left(  \widetilde{a(\ell n)} -  \widetilde{a\left( \frac{n}{\ell}    \right)}  \right) q^{n}.
\end{equation}
We suppose that $\ell \nmid n$. Then we have $\widetilde{a \left( \frac{n}{\ell} \right)} = 0$. We also claim that $\widetilde{a(\ell n)} = 0$. If $\widetilde{a(\ell n)} \neq 0$, then there exists $a, b \equiv 1 \Mod{6}$ such that $\ell n = \frac{a^2 + 2^m b^2}{3}$. The above argument shows that $\ell \mid a, \:b,$ and hence, that $\ell \mid n$,  which is a contradiction. When $\ell \mid n$, we compute 
\begin{align*}
    \widetilde{a(\ell n)} = \sum_{\substack{a, b \: \equiv \: 1 \Mod{6} \\ \frac{a^2 + 2^m b^2}{3}
     = \: \ell n }} \chi_{12}(a b) = \sum_{\substack{r, s \: \equiv \: 1 \Mod{6} \\ \frac{\ell^2 r^2 + 2^m \ell^2 s^2}{3} = \: \ell n }} \chi_{12}((\ell r)(\ell s)) = \sum_{\substack{r, s \: \equiv \: 1 \Mod{6} \\ \frac{r^2 + 2^m s^2}{3} = \frac{n}{\ell} }} \chi_{12}(rs) = \widetilde{a \left( \frac{n}{\ell} \right)}.
\end{align*}  
Using  \eqref{del_act_2}, we conclude that $\widetilde{\Delta}(z)^{\frac{1 + 2^m}{3}} \mid T_{\ell} = 0$.
One can argue similarly to prove the remaining cases in part one of the proposition using the theta-series expansion for $\eta(z)^{3}$ as in~\eqref{eta3_expansion}. Part two of proposition uses the same arguments.

\subsection{Proof of Theorem \ref{partition_5} and Theorem \ref{partition_3}}
We prove Theorem \ref{partition_5} for $p = 5$. One can argue similarly to prove the Theorem for $p \in \{3, 7\}.$

\medskip
The generating function for $a_{t}(n)$ is 
\begin{equation} \label{gen_fun2}
   \sum_{n \geq 0} a_{t}(n) q^{n} = \displaystyle{\prod_{n \geq 1} \frac{(1 - q^{nt})^{t}}{1 - q^{n}}}.
\end{equation}
We recall that $k_{5, t} = \frac{5^{2t} - 1}{24}$.
Using \eqref{eta}, \eqref{delta}, and \eqref{gen_fun2}, we deduce that
\begin{equation*}
\sum_{n \geq 0} \widetilde{a_{5^t}}(n - k_{5, t}) \: q^{n} = \widetilde{\Delta}(z)^{k_{5, t}} \: \text{in} \: \F_{5}[[q]].
\end{equation*}
We let $\ell \equiv \pm 1 \Mod{5}$, and we let $r \geq 1$ have $5^r \geq m_{5, t} = 
 1 + \Big\lfloor \frac{2k_{5, t}}{5} \Big\rfloor$. Part (1) of Theorem~\eqref{simple_bound} implies that 
 \begin{gather*}
0 = \widetilde{\Delta}(z)^{k_{5, t}} \mid (T_{\ell}^{'})^{5^r} =
\begin{dcases}
\widetilde{\Delta}(z)^{k_{5, t}} \mid T_{\ell}^{5^r} &\text{if $\ell \equiv -1 \Mod{5}$}, \quad \\
\widetilde{\Delta}(z)^{k_{5, t}} \mid (T_{\ell}^{5^r} - 2) &\text{if $\ell \equiv 1 \Mod{5}$} \quad
\end{dcases}
 \end{gather*} 
since $(T_{\ell} - 2)^{5^r} \equiv T_{\ell}^{5^r} - 2 \Mod{5}$.
Applying Corollary \ref{coeff_corollary}, we find, for all $n \geq 0$ with $\ell \nmid n$, that
\begin{gather*}
a_{5^t}(\ell^{5^r} n - k_{5, t}) - \ell^{12 k_{5, t} - 1} \: a_{5^t}(\ell^{5^r - 2} \: n - k_{5, t}) \equiv
 \begin{dcases}
0 &\text{if $\ell \equiv -1 \Mod{5}$},\quad \\
2  \widetilde{a_{5^t}}(n - k_{5, t}) &\text{if $\ell \equiv 1 \Mod{5}$},\quad
\end{dcases}
\end{gather*}
holds modulo $5$, 
which proves parts (1) and (3) of Theorem 1.6 for $p = 5$.

While Theorem \eqref{simple_bound} concerns the iteration of a single $T_{\ell}^{'}$, its proof applies more generally to show that 
\begin{equation} \label{Hecke_iterate}
\widetilde{\Delta}(z)^{k_{5, t}} \mid T_{\ell_{1}}^{'} \mid \cdots \mid T_{\ell_{m_{5, t}}}^{'} = 0.
\end{equation}
Since $\ell_{i} \equiv - 1 \Mod{5}$, we have $T_{\ell_{i}}^{'} = T_{\ell_{i}}$. By \eqref{Hecke_defn}, for $n \geq 1$ with $\ell_{1} \nmid n$, we observe that $\widetilde{\Delta}^{k_{5, t}} \mid T_{\ell_{1}}$ has $n$th coefficient 
\[
\test{a_{5^t}(\ell_{1} n - k_{5, t})} - \test{a_{5^t}\left( \frac{n}{\ell_{1}} - k_{5, t} \right)} = \test{a_{5^t}(\ell_{1} n - k_{5, t})}.
\]
Similarly, for $\ell_{1} \neq \ell_{2}$, and $n \geq 1$ with $\ell_{1}, \: \ell_{2} \nmid n$, we find that $\widetilde{\Delta}^{k_{5, t}} \mid T_{\ell_{1}} \mid T_{\ell_{2}}$ has $n$th coefficient 
\[
\test{a_{5^t}(\ell_{1} \ell_{2} n - k_{5, t})} - \test{a_{5^t}\left( \frac{\ell_{1} n}{\ell_{2}} - k_{5, t} \right)} = \test{a_{5^t}(\ell_{1} \ell_{2} n - k_{5, t})}.
\]
We iterate in this way using \eqref{Hecke_iterate} to prove part (2) of Theorem \ref{partition_5} for $p = 5$.

\medskip

To prove Theorem \eqref{partition_3}, we recall that $D_{2}(z) = \eta(2z)^{12}$, and we let $\gcd(r, 6) =~1$. When $n$ is odd, we use \eqref{gen_fun1} to deduce that
\begin{equation*} 
\sum p_{12r} \left(\frac{n - r}{2} \right) q^{n} = \eta(2z)^{12r} = D_{2}(z)^{r}.
\end{equation*}
From here, the proof runs parallel to the proof of Theorem \ref{partition_5} for $p = 5$ given above, this time applying part (2a) of Theorem \ref{simple_bound}.

\end{document}